\documentclass[runningheads]{llncs}
\usepackage[T1]{fontenc}
\usepackage{graphicx}
\usepackage{amsmath}
\usepackage{amssymb}
\usepackage{mathtools}
\usepackage{enumitem}

\spnewtheorem{notation}{Notation}{\bfseries}{}

\newcommand{\Z}{\mathbb{Z}}
\newcommand{\R}{\mathcal{R}}
\newcommand{\set}[2]{\{ #1 \: | \: #2 \}}

\newcommand{\id}{\mathrm{id}}
\newcommand{\cyl}{\mathrm{Cyl}}

\begin{document}
\title{Reversibility, balance and expansivity of non-uniform cellular automata\thanks{Supported by the Magnus Ehrnrooth foundation and the Research Council of Finland.}}
%
%
\author{Katariina Paturi\orcidID{0009-0007-2475-1393}}
\authorrunning{K. Paturi.}
%
\institute{University of Turku, Vesilinnantie 5, 20500 Turku, Finland\\
\email{katariina.paturi@utu.fi}
}
\maketitle              
\begin{abstract}
Non-uniform cellular automata (\textsc{nuca}) are an extension of cellular automata (\textsc{ca}), which transform cells according to multiple different local rules.  A \textsc{nuca} is defined by a configuration of local rules called a local rule distribution. We examine what properties of uniform \textsc{ca} can be recovered by restricting the rule distribution to be (uniformly) recurrent, focusing on only 1D \textsc{nuca}.

We show that a bijective \textsc{nuca} with a uniformly recurrent rule distribution is reversible. We also show that if a \textsc{nuca} is surjective and has a recurrent rule distribution, or if it is bijective, then it is balanced. We present an example of a \textsc{nuca} which has a non-empty and non-residual set of equicontinuity points, and one which is not sensitive but has no equicontinuity points. Finally, we show that (positively) expansive \textsc{nuca} are sensitive.

\keywords{Non-uniform cellular automata \and Symbolic dynamics \and Reversibility.}
\end{abstract}
\section{Introduction}

Non-uniform cellular automata (\textsc{nuca}) are an extension of cellular automata (\textsc{ca}), which are machines that operate on configurations - grids of cells coloured by an alphabet. In both, each cell simultaneously observes its neighbouring cells and changes its state depending on the state of the neighbourhood. A uniform \textsc{ca} transforms all cells with the same local rule, while while a \textsc{nuca} uses different local rules in different cells. The local rules of a \textsc{nuca} are given by a fixed local rule distribution. 

Cellular automata are a powerful tool in simulating systems based on local transformations, and \textsc{nuca} provide a generalization capable of simulating systems with variation in their local behaviour. The dynamical properties of uniform \textsc{ca} have been studied extensively, but many of them no-longer apply in the setting of \textsc{nuca}. For example, the well-known Garden of Eden theorem for uniform \textsc{ca} does not hold for \textsc{nuca}.  However, some properties, such as the Garden of Eden theorem, can be recovered in \textsc{nuca} when the distribution of local rules is recurrent or uniformly recurrent \cite{goe}. 

The properties we examine are reversibility, balance, sensitivity and expansivity. There are implications known to hold for \textsc{ca} which don't necessarily hold for \textsc{nuca} in general: all bijective \textsc{ca} are reversible, all surjective \textsc{ca} are balanced, a \textsc{ca} is either sensitive  or has a dense set of equicontinuity points, and a (positively) expansive \textsc{ca} is sensitive. We find (positively) expansive \textsc{nuca} are still sensitive, but the other statements all have counter-examples. We also show that the reversibility of bijective \textsc{nuca} can be recovered when the rule distribution is uniformly recurrent, and the balance of surjective \textsc{nuca} when the distribution is recurrent. Additionally we show that bijective \textsc{nuca} are balanced. 

\section{Definitions}

Non-uniform cellular automata may be defined over any group, but in this work we only consider the 1-dimensional case of \textsc{nuca} over $\Z$. First, a few background definitions are needed.

\begin{definition}
Let $\Sigma$ be a finite set called a \emph{state set}. A \emph{configuration (on $\Z$)}  is an element $c \in \Sigma^\Z$.
\end{definition}

\begin{definition}
A \emph{cell} is an element $x\in \Z$. A \emph{domain} is any subset $D\subseteq \Z$. If $D$ is finite, we call it a \emph{finite domain}. 

Let $\Sigma$ be a state set and $c \in \Sigma^\Z$ a configuration. A \emph{pattern} is any $p \in \Sigma^D$ for some domain $D$. If $D$ is finite, we call $p$ a \emph{finite pattern}. We use $c_{|D}\in \Sigma^D$ to denote the pattern such that for any $x\in D$, $c_{|D} (x) = c(x)$.
\end{definition}


The space $\Sigma^\Z$ equipped with Cantor's topology and is known to be compact \cite{cecc}. In contrast with uniform \textsc{ca}, which are defined by a local rule, a \textsc{nuca} is defined by a local rule distribution, which is a special configuration that determines which local rule is used in which cell.

\begin{definition}
Let $\Sigma$ be a state set. A \emph{radius $r$ local rule} is a function $f:\Sigma^{2r+1} \rightarrow \Sigma$. A finite set of local rules $\mathcal{\R}$ is called a \emph{rule set}.
\end{definition}

Note that defining local rules with radial neighbourhoods is equivalent to defining them with arbitrarily shaped neighbourhoods. Now we are ready to define \textsc{nuca}.

\begin{definition}
A \emph{non-uniform cellular automaton} (or \textsc{nuca}) is a tuple $A = (\Sigma, \R,\theta)$, where $\Sigma$ is a state set, $\R$ is a rule set, and $\theta \in \R^\Z$ is a rule distribution. The \emph{global transition function} of $A$ is the function $H_\theta:\Sigma^\Z \rightarrow \Sigma^\Z$ such that for $c\in \Sigma^\Z$, $x \in \Z$, 
\begin{align*}
H_\theta (c)(x) = \theta(x) (c(x-r),\ldots,c(x+r)),
\end{align*}
where $r \in \Z_+$ is the radius of the local rule $\theta (x)$. If $|\R| = 1$ then the \textsc{nuca} is a uniform cellular automaton (or \textsc{ca}).
\end{definition}

We generally equate a \textsc{nuca} with its transition function. Let us define a few more concepts used throughout the work. The definitions of recurrence and uniform recurrence are equivalent with the usual topological definitions for subshifts.

\begin{definition}
Let $\theta \in \R^\Z$ be a rule distribution, and suppose $\theta (x)$ is a radius $r$ rule. We call the domain $N_\theta(x)=[x-r,x+r]$ the \emph{neighbourhood of $x$}. If $D \subseteq \Z$ is a domain, we call the domain $N_\theta (D) = \bigcup_{x\in D} N_\theta (x)$ the \emph{neighbourhood of $D$}.
\end{definition}

\begin{definition}
Let $\Sigma$ be a state set. The \emph{left-shift} is the function $\sigma:\Sigma^\Z \rightarrow \Sigma^\Z$ which maps $c \in \Sigma^\Z$ such that $\sigma (c)(x) = c(x+1)$ for all $x\in \Z$. Similarly, we call the function $\sigma^{-1}$ the \emph{right shift}. Note that $\sigma \circ H_\theta = H_{\sigma(\theta)} \circ \sigma$.

If $p\in \Sigma^D$ and $q\in \Sigma^C$ are patterns and there is $k \in \Z$ such that $C = D+k$ and $p(x) = q(x+k)$ for any $x \in D$, then we call $q$ a \emph{translated copy} (or simply \emph{copy}) of $p$. If $D \cap (D+k) = \emptyset$, the copies are \emph{disjoint}.
\end{definition}

\begin{definition}
Let $c \in \Sigma^\Z$ be a configuration. The configuration $c$ is \emph{(spatially) recurrent} if for any finite $D\subseteq \Z$ there is $C \subseteq \Z$ such that $D \neq C$ and $c_{|C}$ is a copy of $c_{|D}$. The configuration $c$ is \emph{(spatially) uniformly recurrent} if for any finite $D\subseteq \Z$ there is $n \in \Z_+$ such that for any $x \in \Z$, there is $C \subseteq [x,x+n]$ such that  $c_{|C}$ is a copy of $c_{|D}$. The same concepts are defined for rules rule distributions $\theta \in  \R^\Z$ when $\R$ is finite.
\end{definition}

\section{Reversibility}

Informally, we call a \textsc{nuca} $H$ reversible if there exists an inverse \textsc{nuca} $H^{-1}$, such that $H\circ H^{-1} = H^{-1} \circ H$ is the identity function. In uniform \textsc{ca} reversibility is equivalent with bijectivity, which in turn is equivalent with injectivity. 

\begin{definition}
Let $\Sigma$ be a state set. By $\id:\Sigma^\Z \rightarrow \Sigma^\Z$ we denote the identity function.
\end{definition}

\begin{definition}
Let $\R_1$ be a rule set and $\theta \in \R_1^\Z$. If there exists finite rule set $\R_2$ and $\phi \in \R_2^\Z$ such that $H_\phi = H_\theta^{-1}$, then the automaton $H_\theta$ is \emph{reversible}.
\end{definition}

Obviously, a reversible \textsc{nuca} is bijective. For uniform \textsc{ca}, the converse requires a non-trivial proof based on the compactness of the configuration space $\Sigma^\Z$. If we allow \textsc{nuca} to have infinitely many rules, the converse is also simple. This follows from the fact that \textsc{nuca} (allowing infinitely many rules) are exactly the continuous functions over the Cantor set \cite{dennuz}, and the fact that the inverse of a bijective continuous function is also continuous. Then the inverse of a bijective \textsc{nuca} is a continuous function over the Cantor set, thus also being a \textsc{nuca}.


It is then natural to ask if this is true restricting ourselves to \textsc{nuca} with finitely many rules only. The following is an example of a bijective \textsc{nuca} which is not reversible.

\begin{example}
We denote $a \oplus b = a + b \mod 2$. Let $\Sigma = \Z_2$ and $\R = \{f_L, f_R, g \}$ be a set of radius $1$ rules over $\Sigma$, which map
\begin{align*}
f_L(a,b,c) &= a \oplus b, \\
f_R(a,b,c) &= b \oplus c, \\
g(a,b,c) &= b,
\end{align*}
where $a,b,c \in \Sigma$. For all $n \in \Z_+$, let $u_n = (f_R)^n g (f_L)^n$ and let $w_n =  u_1 \ldots u_n$. Then, let $\theta = \ldots w_3 w_2 w_1 w_2 w_3 \ldots$ be the rule distribution.

The distribution $\theta$ is clearly recurrent. First, let's show that $H_\theta$ is injective. Let $c,e \in \Sigma^\Z$ be such that $c(x) \neq e(x)$ for some $x \in \Z$. If $\theta (x) = g$, then $H_\theta (c)(x) = c(x)\neq e(x)= H_\theta (e)(x)$. Suppose then $\theta (x) = f_R$. Then there is $k\in \Z_+$ such that  $\theta(x+i) = f_R$ when $0 \leq i < k$ and $\theta(x+k) = g$. 

Now suppose $H_\theta (c)(x+i) = H_\theta (e)(x+i)$ when $0 \leq i < k$. If $c(x+j) \neq e(x+j)$ where $0 \leq j < k$, then $c(x+j+1) \neq e(x+j+1)$ also, because otherwise it would hold that
\begin{align*}
H_\theta (c)(x+j) &= c(x+j) \oplus c(x+j+1) \\
&= c(x+j) \oplus e(x+j+1)\\
&\neq e(x+j) \oplus e(x+j+1) = H_\theta(e)(x+j)
\end{align*}
Now inductively, $c(x+i) \neq e(x+i)$ when  $0 \leq i \leq k$. Then because $c(x+k) \neq e(x+k)$ and $\theta(x+k)=g$, we have $H_\theta(c)(x+k) \neq H_\theta(e)(x+k)$. In any case, $H_\theta (c) \neq H_\theta (e)$.

The case that $\theta(x) = f_L$ is symmetrical. Therefore $H_\theta$ is injective. Because $\theta$ is recurrent, by the Garden of Eden theorem, $H_\theta$ is then also bijective \cite{goe}. Next we show that $H_\theta$ is not reversible with a finite rule set. 

Suppose there is a finite rule set $\R'$ and $\phi \in \R'^\Z$ such that $H_\phi = H_\theta^{-1}$. Then there is some $r\in \Z_+$ such that any rule in $\mathcal{R}'$ is at most radius $r$. Somewhere in $\theta$, there is a length  $2r+1$ segment containing only the rule $f_R$; let $x \in \Z$ be the centermost cell of such a segment. Let then $c,e \in \Sigma^\Z$ be such that $c(y) = 0$ and $e(y) = 1$ for all $y\in \Z$. Then for all $z \in [x-r,x+r]$,
\begin{align*}
H_\theta(c)(z) = 0 \oplus 0 = 1 \oplus 1 = H_\theta(e)(z).
\end{align*}
Suppose $\phi(x)(0,\ldots,0) = a$ for some $a \in \Sigma$. Then 
\begin{align*}
c(x) = (H_\phi \circ H_\theta) (c) (x) = a = (H_\phi \circ H_\theta) (e) (x) = e(x),
\end{align*}
which is a contradiction. Therefore $H_\theta$ is not finitely reversible.
\end{example}

Note that the rule distribution in the previous counter-example is recurrent. A simpler, yet non-recurrent counter-example in the notation of the previous example would be distribution $^\infty (f_R)g(f_L)^\infty$. Next we show that if a bijective \textsc{nuca} is defined by a uniformly recurrent rule distribution, then it is reversible. To do so, we use the notion of stable injectivity. 

\begin{definition}
Let $\theta\in \R^\Z$ be a rule distribution. The \emph{orbit} of $\theta$ is the set $\mathcal{O}(\theta) = \set{\sigma^k (\theta)}{k\in \Z}$. The \emph{orbit closure} $\overline{\mathcal{O}(\theta)}$ is the topological closure of $\mathcal{O}(\theta)$. We call $H_\theta$ \emph{stably injective} if for any $\phi \in \overline{\mathcal{O}(\theta)}$, the \textsc{nuca} $H_\phi$ is injective. \cite{phung}
\end{definition}

\begin{lemma} \label{lemma_inj}
Let $\theta \in \R^\Z$ be a uniformly recurrent rule distribution such that $H_\theta$ is injective. Then $H_\theta$ is stably injective.
\end{lemma}

\begin{proof}
Let $\phi \in \overline{\mathcal{O}(\theta)}$. Suppose $H_\phi$ is not injective. Then there are two cases: suppose first that $H_{\phi}$ is not pre-injective. Let then $c'',e'' \in \Sigma^\Z$ be asymptotic configurations such that $c''\neq e''$ and $H_{\phi} (c'') = H_{\phi} (e'')$, and $x,y\in \Z$ such that for any cell $z \notin [x,y]$, $c''(z) = e''(z)$. Because $\phi \in \overline{\mathcal{O}(\theta)}$, any pattern in $\phi$ has a copy in $\theta$. Denote $D=[x-r,y+r]$ and let $m \in \Z$ be such that $\theta_{|(D+m)}$ is a copy of $\phi_{|D}$. Then because for any cell $z \in (D+m)$, $\theta(z)=\phi(z-m)$, so 
\begin{align*}
H_\theta (\sigma^{-m} (c''))(z) = H_{\phi} (c'')(z-m) = H_{\phi} (e'')(z-m) =H_\theta (\sigma^{-m} (e''))(z),
\end{align*}
and for any cell $z \notin (D+m)$, $\sigma^{-m} (c'')_{|[z-r,z+r]}=\sigma^{-m} (e'')_{|[z-r,z+r]}$, we have $H_\theta (\sigma^{-m} (c'')) = H_\theta (\sigma^{-m} (e''))$. Then $H_\theta$ is not pre-injective, which is a contradiction.

Suppose then that $H_{\phi}$ is pre-injective. Then either every length $2r+1$ segment to the right of the origin 0 contains a cell where $c'$ and $e'$ differ, or the same holds for every $2r+1$ segment to the left of the origin. Assume without loss of generality the first case, that is for every $k>0$ there is $i \in [-r,r]$ such that $c'(k+i) \neq e'(k+i)$.

Because $\theta$ is uniformly recurrent, so is $\phi$ and $\theta \in \overline{\mathcal{O} (\phi)}$. Hence every finite segment of $\theta$ appears in $\phi$ in every segment of some length $n>0$. Then there exists $k_1, k_2, k_3, \ldots > 0$ such that $\sigma^{k_1} (\phi),\sigma^{k_2} (\phi),\ldots $ converges to $\theta$. Select then a subsequence  $k_{i_1},k_{i_2},\ldots$ such that $c'^{k_{i_1}},c'^{k_{i_2}},\ldots$ and $e'^{k_{i_1}},e'^{k_{i_2}},\ldots$ converge to some $c''$ and $e''$ respectively. Then by the previous we have $c''_{|[-r,r]}\neq e''_{|[-r,r]}$, and by continuity of \textsc{nuca}, $H_\theta (c'') = H_\theta (e'')$. Then $H_\theta$ is not injective, which is a contradiction. 
\qed
\end{proof}

Note that Lemma \ref{lemma_inj} doesn't hold in general, even in $\Z^2$. Since bijective, stably injective NUCA are reversible \cite{phung}, we can prove the following theorem.

\begin{theorem} \label{unif_revers_fin}
Let $\R$ be a finite rule set and $\theta\in \R^\Z$ a uniformly recurrent rule distribution such that $H_\theta$ is injective. Then $H_\theta$ is reversible.
\end{theorem}

\begin{proof}
 By Lemma \ref{lemma_inj}, $H_\theta$ is stably injective. Because $\theta$ is recurrent and $H_\theta$ is injective, then $H_\theta$ is bijective \cite{goe}. Then there is a finite rule set $\mathcal{R}_2$ and $\phi \in \R_2^\Z$, such that $H_\phi \circ H_\theta = \id$ \cite{phung}. Now because $H_\theta$ is surjective,  we have $H_\phi = H_\theta^{-1}$.

\qed
\end{proof}
 
We can also show that the inverse automaton of a reversible \textsc{nuca} with a uniformly recurrent rule distribution will also have a uniformly recurrent rules distribution, up to renaming rules.

\begin{definition}
Let $f:\Sigma^{2r+1} \rightarrow \Sigma$ and $g:\Sigma^{2R+1} \rightarrow \Sigma$ be radius $r$ and $R$ local rules respectively for some $r,R\in \Sigma^\Z$. Assume $r \leq R$. The rules $f$ and $g$ are considered \emph{identical}, if for any $a_{-R},\ldots,a_{R} \in \Sigma$ it holds that
\begin{align*}
g(a_{-R},\ldots,a_{R}) = f(a_{-r},\ldots,a_r).
\end{align*}
\end{definition}

\begin{theorem} \label{unif_revers_unif}
Let $\R_1$ be a finite rule set and $\theta \in \R_1^\Z$ uniformly recurrent. Let $\R_2$ be a rule set containing no pair of different identical rules, and $\phi \in \R_2^\Z$ a distribution such that $H_\phi = H_\theta^{-1}$. Then $\phi$ is uniformly recurrent.
\end{theorem}

\begin{proof} 
Let $r\in \Z_+$ be at least as large as the radius of any rule in $\R_1$ or $\R_2$. We show that for any $x,y \in \Z$, if $\theta_{|[x-r,x+r]} = \sigma^k (\theta_{|[y-r,y+r]})$, where $k=y-x$, then $\phi(x) = \phi(y)$. For all $c \in \Sigma^\Z$ we have $H_\theta (c)_{[x-r,x+r]} = H_{\sigma^k (\theta)} (c)_{[x-r,x+r]}$. Then applying the local rule $\phi(x)$, we have
\begin{align*}
H_\phi (H_\theta (c)) (x) = H_\phi (H_{\sigma^k (\theta)} (c)) (x).
\end{align*}
Since $H_\phi \circ H_\theta = \id = H_{\sigma^k (\phi)} \circ H_{\sigma^k (\theta)}$, we then have
\begin{align*}
H_{\sigma^k(\phi)} (H_{\sigma^k (\theta)} (c)) (x) = H_\phi (H_{\sigma^k (\theta)} (c)) (x).
\end{align*}
Because $H_{\sigma^k (\theta)}$ is surjective, for any $e \in \Sigma^\Z$ there is $c \in \Sigma^\Z$ such that $H_{\sigma^k (\theta)}(c) = e$. Therefore for any $e\in \Sigma^\Z$, 
\begin{align*}
H_{\sigma^k(\phi)} (e) (x) = H_\phi (e) (x).
\end{align*}
Since there are no identical rules in $\R_2$, $\sigma^k(\phi)(x) = \phi(x)$. Now because $\theta$ is uniformly recurrent, for any $m \in \Z_+$, and any segment $D \subseteq \Z$ of length $|D|=m+2r$, there is $n \in \Z_+$ such that  a copy of $\theta_{|D}$ appears in every length $n$ segment of $\theta$. Then by the previous, for every segment $D'$ of length $|D'|=m$ there is $n\in \Z_+$ such that a copy of $\phi_{|D'}$ appears in every segment of length $n$. Hence $\phi$ is uniformly recurrent.
\qed

\end{proof}

\section{Balance}

Balance is a property found in surjective \textsc{ca}, where every finite pattern in a given domain has the same number of pre-images. Balance is equivalent to invariance of the uniform measure. Surjective \textsc{nuca} are not necessarily balanced, as shown by the following example.

\begin{definition}
Let $\theta \in \R^{\Z^d}$ be a rule distribution with state set $\Sigma$. Let $D \subseteq \Z^d$ be a finite domain. The function $H_{\theta |D}: \Sigma^{N_\theta(D)} \rightarrow \Sigma^D$ is a \emph{non-uniform \textsc{ca} over finite domain} $D$, which maps $c_{|N_\theta (D)}$, where $c \in \Sigma^\Z$, to
\begin{align*}
H_{\theta|D} (c_{|N_\theta (D)}) = H_\theta (c)_{|D}.
\end{align*}
\end{definition}

\begin{definition}
Let $H_\theta$ be a \textsc{nuca}. We say $H_\theta$ is \emph{balanced}, when for any finite domain $D\subseteq \Z$ and pattern $p \in \Sigma^D$, if $E = N_\theta(D)$ and $\mathcal{A} = \set{q \in \Sigma^E}{H_{\theta|D}(q) = p}$, then $|\mathcal{A}| = |\Sigma|^{|E|-|D|}$.
\end{definition}


\begin{example}
Let $\Sigma = \Z_2$ and $\R = \{f,g\}$ be a set of radius $1$ rules, where
\begin{align*}
f(a,b,c) = a \oplus b \oplus c, \:
g(a,b,c) = \max (b,c), 
\end{align*}
Let $\theta \in \R^\Z$ be the rule distribution such that
\begin{align*}
\theta (x) &= \begin{cases}
g, \text{ if } x=0,\\
f, \text{ if } \text{otherwise.}
\end{cases}
\end{align*}
First, let's show $H_\theta$ is surjective. Let $c \in \Sigma^\Z$ be any configuration. Then let $e \in \Sigma^\Z$ be a configuration such that 
\begin{align*}
e(x) = \begin{cases}
0, \text{ if } x=-1 \text{ or } x=1\\
c(x), \text{ if } x=0,\\
c(x+1) \oplus e(x+1) \oplus e(x+2), \text{ if } x < -1,\\
c(x-1) \oplus e(x-1) \oplus e(x-2), \text{ if } x>1,
\end{cases}
\end{align*}
where $x \in \Z$. Now if $x \neq 0$ clearly $H_\theta (e)(x) = c(x)$, and if $x = 0$, $H_\theta (e)(x) = \max (c(x), 0) = c(x)$. Therefore $H_\theta (e) = c$, meaning $H_\theta$ is surjective.

On the other hand, let $D =\{0\}$ and $p' \in \Sigma^D$ be such that $p'(0) = 1$. Now the patterns $p \in \Sigma^{N_\theta(D)}$ such that $H_{\theta|D} (p) = p'$ are $001,010,011,101,110$ and $111$, for a total of $t = 6$ patterns. Then clearly $t = 6 \neq 4 = 2^2 =  |\Sigma|^{|N(D)|-|D|}$. 
\end{example}

Note that the \textsc{nuca} in the previous example is also pre-injective, that is, all configurations which differ in only finitely many cells have different images. Unlike in uniform \textsc{ca}, pre-injectivity and surjectivity aren't necessarily equivalent in \textsc{nuca} \cite{goe}.

It can be shown that a surjective \textsc{nuca} of any dimension with a uniformly recurrent rule distribution has the balance property, with a proof very similar to the proof on uniform \textsc{ca}. In the one-dimensional case, it is enough for the rule distribution to be recurrent. Note that surjectivity and pre-injectivity are equivalent in \textsc{nuca} with recurrent rule distributions \cite{goe}.

\begin{lemma} \label{ineq2} 
Let $s,n,r \in \Z_+$. For all sufficiently large $m \in \Z_+$, it holds that
\begin{align*}
(s^n-1)^m s^{k-nm} < s^{k-2r},
\end{align*}
for all $k \in \Z_+$. 
\end{lemma} 

\begin{theorem}
Let $\theta \in \mathcal{R}^\Z$ be a recurrent rule distribution such that $H_\theta$ is surjective. Then $H_\theta$ is balanced.
\end{theorem}

\begin{proof}
Let $D\subseteq \Z$ be a finite domain and $E=N_\theta (D)$. Without loss of generality, we can assume all rules in $\R$ have radius $r$. Suppose there is $p \in \Sigma^D$ such that the number of $H_{\theta|D}$ pre-images is $t\neq s^{|E|-|D|}$. We can select $p$ such that $t< s^{|E|-|D|}$, because if every pattern in domain $D$ were to have at least $s^{|E|-|D|}$ patterns, since there are only $s^E$ patterns in domain $E$, every pattern would have exactly $s^{|E|-|D|}$ pre-images, contradicting our assumption.

Because $\theta$ is recurrent, it contains infinitely many copies of $\theta_{|D}$. Let $n \in \Z_+$ be such that $|D|=n-2r$. Let then $m\in \Z_+$ be large enough that $(s^n-1)^m s^{k-nm}<s^{k-2r}$ for any $k\in \Z_+$. By Lemma \ref{ineq2} such a number exists.

Let $D'$ be a $k-2r$ wide segment where $k\in \Z_+$ is large enough that a $\theta_{|D'}$ contains $m$ disjoint copies of $\theta_{|D}$ that are at least $2r$ apart, and let $E'= N_\theta (D')$. Of course now $|E'|=k$. Let $\mathcal{B} \subseteq \Sigma^{D'}$ be the set of patterns such that there is a copy of $p$ on each disjoint copy of $\theta_{|D}$. Because the rest of the cells may be freely chosen, $|\mathcal{B}|=s^{|D'|-m|D|} = s^{k-2r-m(n-2r)}$.

Let then $\mathcal{A} = \set{q \in \Sigma^{E'}}{H_{\theta|D'}(q) \in \mathcal{B}}$. Then $|\mathcal{A}|=t^m \cdot s^{|E'|-m|E|} = t^m \cdot s^{k-mn}$, because on each $n$ wide segment centered the copies of $\theta_{|D}$ must be a copy of one of the $t$ pre-images of $p$, and the other cells may be chosen freely. Now it remains to show that $|\mathcal{A}| < |\mathcal{B}|$. Now because $t\leq s^{2r}-1$, using the previous, we have 
\begin{align*}
|\mathcal{A}| &= t^m \cdot s^{k-mn} \leq (s^{2r}-1)^m \cdot s^{k-mn}\leq (s^{2r}- s^{-(n-2r)})^m\cdot s^{k-mn}\\
&= (s^n - 1)^m  \cdot s^{k-mn} \cdot s^{-m(n-2r)} < s^{k-2r} \cdot s^{-m(n-2r)} = |\mathcal{B}|.
\end{align*}
Hence $H_\theta$ is not surjective. Then if $H_\theta$ is surjective, the theorem statement holds.
\qed
\end{proof}

Though surjectivity is not enough to guarantee balance, bijectivity is. In fact, it is enough to assume that every configuration has the same number of pre-images. This also holds for any dimension, though we do not show it here.

\begin{theorem} \label{bij_bal}
Let $\theta \in \mathcal{R}^\Z$ be a rule distribution such that $H_\theta$ is bijective. Then $H_\theta$ is balanced.
\end{theorem}

\begin{proof}
Let $D\subseteq \Z$ be a finite segment and $p \in \Sigma^D$ a pattern. Let $E = N_\theta(D)$ and $\mathcal{A}_p = \set{q \in \Sigma^E}{H_{\theta|D}(q) = p}$.
Without loss of generality, we can assume $D$ is a contiguous segment. Let $D' = \Z \setminus N_\theta(E)$ and $E' = \Z \setminus E$. Let $C = \Z \setminus (D \cup D')$. Let $p' \in \Sigma^{D'}$ be a fixed (infinite) pattern and let $\mathcal{B}_{p'} = \set{q' \in \Sigma^{E'}}{H_{\theta|D'} (q') = p'}$. Let $p \in \Sigma^D$ be a finite pattern. Then consider $\mathcal{C}_{p,p'} = \set{c \in \Sigma^\Z}{c_{|D} = p, c_{|D'} = p' }$ and $\mathcal{C}_{p,p'}' = \set{c' \in \Sigma^\Z}{H_\theta(c') \in \mathcal{C}_{p,p'}}$. Obviously $|\mathcal{C}_{p,p'}| = |\Sigma|^{|C|}$. On the other hand, we have $|\mathcal{C}_{p,p'}'|=|\mathcal{A}_p|\cdot |\mathcal{B}_{p'}|$. Finally, because $H_\theta$ is bijective, we have $|\mathcal{C}_{p,p'}| = |\mathcal{C}_{p,p'}'|$. Therefore
\begin{align*}
|\mathcal{A}_p|\cdot |\mathcal{B}_{p'}| = |\mathcal{C}_{p,p'}'| = |\mathcal{C}_{p,p'}| = |\Sigma|^{|C|},
\end{align*}
which gives $|\mathcal{A}_p| = \frac{|\Sigma|^{|C|}}{|\mathcal{B}_{p'}|}$. Choice of $p'$ and $C$ was independent of the choice of $p$, and therefore this is true for any $p \in \Sigma^D$, and therefore every pattern in the domain $D$ has the same number of pre-images. Hence we have  $|\mathcal{A}_p| = |\Sigma|^{|E|-|D|}$.
\qed
\end{proof}

\begin{figure}
\begin{center}
	\includegraphics[scale=0.9]{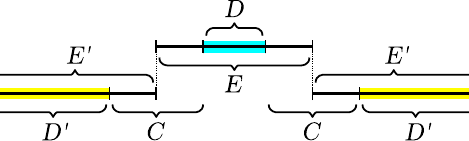}
	\caption{Naming scheme of the domains used in the proof of Theorem \ref{bij_bal}.} \label{fig_bij_bal}
\end{center}
\end{figure}

\section{Sensitivity and equicontinuity}

Sensitivity and equicontinuity are properties of general dynamical systems. To define them, some topological concepts are needed. 

\begin{definition}
Let $\Sigma$ be a state set, $c \in \Sigma^\Z$ a configuration and $D \subseteq \Sigma$ a finite domain. We denote a \emph{cylinder} as
\begin{align*}
\mathrm{Cyl} (c,D) = \set{e \in \Sigma^\Z}{c_{|D} = e_{|D}},
\end{align*}
when $\Sigma$ is clear from context.
\end{definition}

It's well known that cylinders form a topological basis for the Cantor space \cite{kari}, that is, every open set is a union of cylinders. Topologically, a set $A\subseteq \Sigma^\Z$ is dense if every cylinder contains an element of $A$, and residual if it is a countable intersection of dense open sets. Because the Cantor space is a Baire space, all residual sets are also dense. 

The following definitions are equivalent with the topological definitions of equicontinuity and sensitivity. They can naturally be extended to any dimension, but here we focus on only 1-dimensional \textsc{nuca}.

\begin{definition}
Let $H_\theta$ be a \textsc{nuca} with the state set $\Sigma$. The configuration $c \in \Sigma^\Z$ is an \emph{equicontinuity point} if for every finite $E \subseteq \Z$ there is a finite $D\subseteq \Z^d$ such that
\begin{align*}
\forall  e \in \mathrm{Cyl}(c,D)  , n\in \Z_+: H_\theta^n (e) \in \mathrm{Cyl}(H_\theta^n(c),E).
\end{align*}
We denote the set of equicontinuity points as $\mathcal{E}$. We say $H_\theta$ is \emph{equicontinuous}, if $\mathcal{E} = \Sigma^\Z$. If $\mathcal{E}$ is residual, we say $H_\theta$ is \emph{almost equicontinuous}.
\end{definition}

In uniform \textsc{ca} we find that only eventually temporally periodic \textsc{ca} are equicontinuous. However, there are (reversible) equicontinuous \textsc{nuca} which have no periodic configurations \cite{kamilya}.

\begin{definition}
Let $H_\theta$ be a \textsc{nuca} with the state set $\Sigma$. The \textsc{nuca} is called \emph{sensitive} if there exists a finite $E\subseteq \Z$ such that for every $c \in \Sigma^\Z$ and every finite $D \subseteq \Z$,
\begin{align*}
\exists  e\in \mathrm{Cyl} (c,D),n \in \Z_+ :  H_\theta^n (e) \notin \mathrm{Cyl}(H_\theta^n(c),E).
\end{align*}
\end{definition}

In 1-dimensional uniform \textsc{ca}, we find that either $\mathcal{E} = \emptyset$, which is equivalent with sensitivity, or $\mathcal{E}$ is residual. 
However, there are 2-dimensional uniform \textsc{ca} whose set of equicontinuity points is neither empty nor residual, as well as ones which have no equicontinuity points but are not sensitive. \cite{equisense}

In the 1-dimensional non-uniform case, the dichotomy found in uniform \textsc{ca} does not hold, and in fact there exist both 1D \textsc{nuca} whose set of equicontinuity points is neither empty nor residual, and 1D \textsc{nuca} which have no equicontinuity points but are not sensitive. 

\begin{theorem} \label{equicounter}
There is a rule set $\R$ and rule distribution $\theta \in \R^\Z$ such that if $\mathcal{E}$ is the set of equicontinuity points of $H_\theta$, then $\mathcal{E} \neq \emptyset$ and $\mathcal{E}$ is not residual.
\end{theorem}

\begin{proof}
Let $\Sigma = \{0,1\}$ and $\R = \{\mathrm{id}, \tau\}$ where $\mathrm{id}$ is the identity rule and and $\tau$ is the leftward traffic rule, a radius $1$ which maps
\begin{align*}
\tau (a,b,c) = \begin{cases}
0, \text{ if } a=0, b=1, \\
1, \text{ if } b=0, c=1, \\
b, \text{ otherwise,}
\end{cases}
\end{align*}
for any $a,b,c\in\Sigma$. Let $\theta \in \R^\Z$ be such that $\theta(x) = \tau$ if $x> 0$ and $\theta(x)=\mathrm{id}$ if $x\leq 0$. Let $\mathcal{E}$ be the set of equicontinuity points of $H_\theta$. The operation of $H_\theta$ is illustrated in Figure \ref{fig_equi1}.

\begin{figure}
\begin{center}
\includegraphics[scale=0.7]{"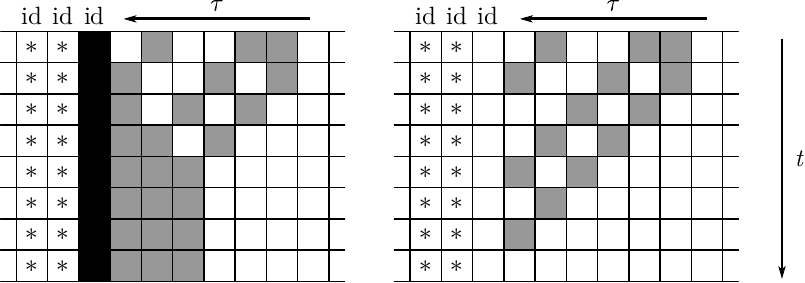"}
\caption{Space-time diagram of the \textsc{nuca} defined in the proof of Theorem \ref{equicounter}. On the left, cell $0$ is set to "1" and doesn't let traffic pass, whereas on the right its set to "0" and annihilates traffic.} \label{fig_equi1}
\end{center}
\end{figure}


First to show that $\mathcal{E} \neq \emptyset$. Let $c_1 \in \Sigma^\Z$ be a configuration of all 1's. Because every cell in $c_1$ sees only 1's, clearly $H_\theta (c_1) = c_1$. Let $E \subseteq \Z$ be a finite segment and let $k\in \Z$ be the rightmost cell of $E$ and let $D = E \cup [0,k]$. Let then $e \in \mathrm{Cyl}(c_1,D)$. For all $x\leq 0$ obviously $H_\theta^n (e) (x) = e(x)$ for any $n \geq 1$. Then if $k\leq 0$, we have $H_\theta^n (e) (x) = e(x) = c(x) = H_\theta^n (c)(x) $ for all $x \in E$. 

Suppose then $k> 0$. For any cell $x>0$, if $e(x)=1$, then $H_\theta(e)(x)=0$ only if $e(x-1) = 0$. However because for all $x \in [0,k]$, $e(x)=1$, and by the previous $H_\theta(e)(0)=1$, then also $H_\theta(e)(x) = 1$. Now we have $H_\theta (e) \in \mathrm{Cyl} (c_1, D)$ and because $e$ was arbitrary, then $H_\theta (\mathrm{Cyl} (c_1, D) ) \subseteq \mathrm{Cyl} (c_1, D)$ and hence $H_\theta ^n (e) \in \mathrm{Cyl} (c_1, D) \subseteq \mathrm{Cyl} (c_1, E)$ for all $n>0$. Therefore $c_1 \in \mathcal{E}$.\\

Now to show that $\mathcal{E}$ is not residual. Because residual sets in a Baire space are dense, it is enough to show that $\mathcal{E}$ is not dense. Let $c \in \Sigma^\Z$ be any configuration such that $c(0) = 0$. We show that $c \notin \mathcal{E}$. Let $E = \{1\}$ and $D \subseteq \Z$ be any finite domain. Without loss fo generality we can assume $1 \in D$. Informally, the idea is that the leftward traffic is annihilated at cell $0$ and traffic from arbitrarily far to the left can then arrive at $E$. 

Let $e_0,e_1\in \mathrm{Cyl}(c,D)$ be such that for all $x \in \Z \setminus D$, $e_0(x) = 0$ and $e_1(x) = 1$. Firstly, there are finitely many $x>0$ such that $e_0(x) = 1$. Also the number of such cells can only decrease, since under the rule $\tau$, when a "0" cell turns into a "1" cell, another "1" cell turns into a "0" cell. Then for any $t\in \Z_+$ such that $H_\theta^t (e_0)(1) = 1$, the number of cells $x>0$ such that $H_\theta^{t+1} (e_0)(x) = 1$ is greater than the number of such cells with $H_\theta^{t} (e_0)(x) = 1$. On the other hand if for some $k>0$ and $H_\theta^t (e_0)(i) = 0$ for all $i \in [1,k]$ and $H_\theta^t (e_0)(i) = 1$ and $H_\theta^t (e_0)(k+1)=1$, then $H_\theta^{t+k} (e_0)(1) = 1$. Hence there is some $n\geq 0$ such that for all $t\geq n$, $H_\theta^{t} (e_0)(1) = 0$, because the number of positive "1" cells decreases to 0.

On the other hand, there are infinitely many $t \geq 0$ such that $H_\theta^t (e_1)(1) = 1$, because there are infinitely many "1" cells to the right of 1.   Now there must be some $t\geq 0$ such that $H_\theta^t(e_0)(1) \neq H_\theta^t(e_1)(1)$. Therefore either $H_\theta^t (e_0) \notin \mathrm{Cyl}(c,E)$ or $H_\theta^t(e_1) \notin \mathrm{Cyl}(c,E)$. In either case, because $D$ was arbitrary, this means $c$ is not an equicontinuity point. Now since no configuration with "0" at cell 0 is an equicontinuity point, $\mathcal{E}$ is not dense, and therefore not residual.



\qed
\end{proof}

\begin{theorem} \label{equicounter2} 
There is a rule set $\mathcal{R}$ and rule distribution $\theta \in \mathcal{\Z}$ such that $H_\theta$ is not sensitive, and if $\mathcal{E}$ is the set of equicontinuity points of $H_\theta$, then $\mathcal{E} = \emptyset$.
\end{theorem}

\begin{proof}
Let $\Sigma = \{0,1,2,3 \}$ and $\mathcal{R} = \{\id, f \}$ a rule set, where $\id$ is the identity rule, and $f:\Sigma^3 \rightarrow \Sigma$ is a radius-1 rule that maps
\begin{align*}
&(a,0,3)\mapsto 3, (1,1,1)\mapsto 1, (1,1,2)\mapsto 1, (a,1,3)\mapsto 3,\\
&(1,2,c)\mapsto 2, (a,3,3)\mapsto 3, \text{ and } (a,b,c)\mapsto 0 \text{ otherwise},
\end{align*} 
where $a,b,c\in \Sigma$. Let $\theta \in \mathcal{R}^\Z$ be such that $\theta (x) = \mathrm{id}$ if $x\leq 0$ and $\theta (x) = f$ if $x>0$. The operation of this \textsc{nuca} is illustrated in Figure \ref{fig_equi2}.

\begin{figure}
\begin{center}
\includegraphics[scale=0.8]{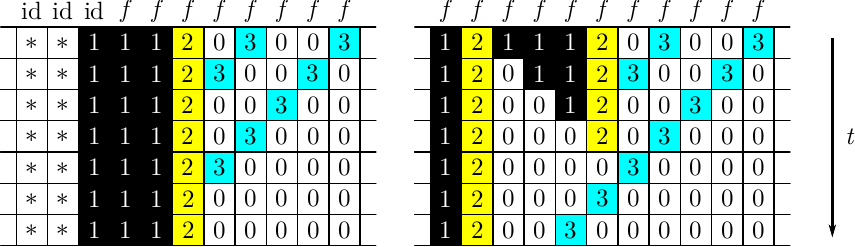}
\caption{Space-time diagrams of the \textsc{nuca} defined in the proof of Theorem \ref{equicounter2}. The left diagram illustrates how a segment arbitrarily far from the origin can be "protected" from 3-cells with a blocking pattern, and the right diagram illustrates why said blocking pattern can't be copied to protect every cell from 3-cells.} \label{fig_equi2}
\end{center}
\end{figure}

First to show that $H_\theta$ is not sensitive: let $E\subseteq [i,j] \subseteq \Z$ be an arbitrary finite domain. Let $c\in \Sigma^\Z$ be such that $c(j) = 2$ and $c(x) = 1$ if $x\neq j$, and let then $D=[\min (i,0),j]$. Let $e \in \cyl (c,D)$. If $x\in[i, 0]$, then obviously $H_\theta (e)(x) = c(x)$, because $\theta (x) = \id$. For every $x \in [1,j-1]$, $H_\theta (e)(x) = f(1,1,a) = 1 = c(x)$, where $a \in \{1,2\}$ and $H_\theta (e) (j) = f(1,2,b) = 2 = c(j)$, where $b \in \Sigma$. Then $H_\theta (e) \in \cyl (c,D)$, and because $e$ was an arbitrary element of $\cyl (c,D)$, then $H_\theta ^n (e) \in \cyl (c,D)$ for all $n>0$. Since also $c \in \cyl (c,D)$, and $E \subseteq D$, we have $H_\theta ^n (e) \in \cyl ( H_\theta^n (c),D )\subseteq \cyl (H_\theta ^n (c), E)$ for all $n>0$. Therefore $H_\theta$ is not sensitive.

Then to show that $\mathcal{E} = \emptyset$. First we show that if a "2"-cell (on the right side) has any non-"1"-cell between it and the origin, it will eventually become a "0"-cell. 
Let $c \in \Sigma^\Z$ be a configuration with  $x_1,x_2 \in \Z$, $0\leq x_1< x_2$ such that $c(x_1)\neq 1$, $c(x_2)=2$ and $c(x) = 1$ when $x_1 < x < x_2$. Suppose for some $n \in [0, x_2-x_1-2]$ we have 
\begin{align*}
H_\theta ^n (c)(x_1+n) \neq 1, \:
H_\theta ^n (c)(x_2) = 2, \:
H_\theta ^n (c)(x) = 1,
\end{align*}
if $x_1+n < x < x_2$. Then under the local rule $f$,
\begin{align*}
H_\theta ^{n+1} (c)(x_1+n+1) = 0,\:
H_\theta ^{x+1} (c)(x_2) = 2, \:
H_\theta ^{n+1} (c)(x) = 1,
\end{align*}
if $x_1+n+1 < x < x_2$. Then by induction, $H_\theta ^{x_2-x_1-1}(c)(x_2-1)=0$ and $H_\theta ^{x_2-x_1-1}(c)(x_2)=2$, and hence $H_\theta ^{x_2-x_1}(c)(x_2)=0$.

Let now $c \in \Sigma^\Z$. We show that $c \notin \mathcal{E}$. Let $x\geq 1$ be the smallest integer such that $c(x-1)=2$. If no such cell exists, we let $x=0$. Let $E=\{x\}$, and let $D=[y,z]$ with any $y,z \in \Z$, $y\leq z$. Without loss of generality, we can assume $x \in D$. Let $e_1,e_2 \in \cyl (c,D)$ be such that if $k \notin D$, then $e_1(k) = 0$ and $e_2(k) = 3$.

By the previous, because no non-"2"-cell can become a "2"-cell, there is $n\geq 0$ such that for all $k \in [x,z]$, $H_\theta ^n (e_1)(k) \neq 2$ and $H_\theta ^n (e_2)(k) \neq 2$. Then there are only finitely many "3"-cells to the right of $x$ in $H_\theta ^n (e_1)$, so for some $n_1 \geq n$, for all $t\geq n_1$ we have $H_\theta ^t (e_1)(x) \neq 3$. On the other hand, because there are infinitely many "3"-cells to the right of $x$ in $H_\theta ^n (e_2)$, for some $n_2 \geq n$, for infinitely many $t > n_2$, $H_\theta ^t (e_2)(x) = 3$. Then there is some $t>0$ such that $H_\theta ^t (e_1)(x) \neq 3 = H_\theta ^t (e_2)(x)$, and therefore either $H_\theta ^t (e_1) \notin \cyl (H_\theta ^t (c),E)$ or $H_\theta ^t (e_2) \notin \cyl (H_\theta ^t (c),E)$. Since this is true for any finite segment $D$, $c \notin \mathcal{E}$, and therefore $\mathcal{E} = \emptyset$.
\qed
\end{proof}

It is still open whether there exist counter-examples that have uniformly recurrent rule distributions.

\section{Expansivity}

In \textsc{ca}, expansivity is a much stronger notion of sensitivity. A \textsc{(nu)ca} is expansive if there's a fixed window where any difference between configurations propagates.

\begin{definition}
Let $\theta \in \mathcal{R}^\Z$ be a rule distribution. The \textsc{nuca} $H_\theta$ is \emph{positively expansive} if there's some finite domain $E \subseteq \Z$ such that for any $c,e\in \Sigma^\Z$, $c \neq e$, there is $n\in \Z_+$ such that $H_\theta ^n (e) \notin \cyl (H_\theta ^n (c),E)$.

The \textsc{nuca} $H_\theta$ is \emph{expansive} if it's bijective and there's a finite domain $E \subseteq \Z$ such that for any $c,e\in \Sigma^\Z$, $c \neq e$, there is $n\in \Z$ such that $H_\theta ^n (e) \notin \cyl (H_\theta ^n (c),E)$.
\end{definition}

In uniform \textsc{ca}, the usual proof that expansivite \textsc{ca} are sensitive utilizes the denseness of periodic points. In bijective \textsc{nuca}, periodic points are not necessarily dense \cite{kamilya}, so we use the Poincaré recurrence theorem instead. Let's first recall some definitions from ergodic theory.

\begin{definition}
Let $\Sigma$ be a state set. A Borel measure $\mu: \mathcal{B} \rightarrow \mathbb{R}$, where $\mathcal{B}$ is a Borel algebra over $\Sigma^\Z$, is \emph{uniform} if it maps 
\begin{align*}
\cyl (c,D) \mapsto \left( \frac{1}{|\Sigma|} \right) ^{|D|},
\end{align*}
where $c \in \Sigma^\Z$ and $D\subseteq \Z$ is a finite domain.\\ 

Let $H_\theta:\Sigma^\Z \rightarrow \Sigma^\Z$ be a \textsc{nuca}. We denote $H_\theta(\mu):\mathcal{C}\rightarrow \mathbb{R}$ as the measure that gives $H_\theta(\mu)(A) = \mu (H_\theta^{-1} (A))$ for any $A \in \mathcal{B}$. We say that measure $\mu$ is \emph{invariant under} $H_\theta$ if $H_\theta(\mu) = \mu$.
\end{definition}

It's well known that the value of a measure on cylinders uniquely determines it, that is, we need to only define a measure on cylinders. \cite{ergodic}
The following proof is essentially identical to the proof on uniform automata.
 
\begin{lemma}
Let $H_\theta:\Sigma^\Z \rightarrow \Sigma^\Z$ be a bijective \textsc{nuca}. The uniform probability measure $\mu$ is invariant under $H_\theta$.
\end{lemma}

\begin{proof}
Let $C = \cyl (c, D)$ for some finite domain $D \subseteq \Z$ and $c \in \Sigma^\Z$. By Theorem \ref{bij_bal},
\begin{align*}
H_\theta^{-1} (C) = C_1 \cup C_2 \cup \ldots \cup C_n,
\end{align*}
where $C_i$ are disjoint cylinders with some finite domain $E \subseteq \Z$, and $n=|\Sigma|^{|E|-|D|}$. Then due to the additivity of measures, \cite{ergodic}
we have
\begin{align*}
\mu (H_\theta^{-1} (C) ) = \mu (C_1) + \mu (C_2) + \ldots + \mu(C_n) = n \left( \frac{1}{|\Sigma|} \right)^{|E|} = \left( \frac{1}{|\Sigma|} \right)^{|D|} = \mu (C).
\end{align*}
Therefore $H_\theta (\mu) = \mu$.
\qed
\end{proof}

It is well know that if a dynamical system has an invariant measure of full support, recurrent points in it are dense. Therefore we state the following lemma without proof. \cite{ergodic}

\begin{definition}
Let $H_\theta:\Sigma^\Z \rightarrow \Sigma^\Z$ be a \textsc{nuca}. A configuration $c \in \Sigma^\Z$ is \emph{temporally recurrent} if for any $D \subseteq \Z$ there exists $n\in \Z_+$ such that $H_\theta^n (c) \in \cyl (c,D)$.
\end{definition}

\begin{lemma}[Poincaré]\label{poincare}
Let $H_\theta:\Sigma^\Z \rightarrow \Sigma^\Z$ be a bijective \textsc{nuca}. Then temporally recurrent configurations of $H_\theta$ are dense, that is, for any cylinder $C$ there is a configuration $c \in C$ that is temporally recurrent under $H_\theta$. 
\end{lemma}

Now we are ready to show that expansive \textsc{nuca} are indeed sensitive.

\begin{theorem}
Let $\theta \in \mathcal{R}^\Z$ be a rule distribution such that $H_\theta$ is expansive or positively expansive. Then $H_\theta$ is sensitive.
\end{theorem}

\begin{proof}
If $H_\theta$ is positively expansive, it is clearly sensitive by definition. Suppose then that (bijective) $H_\theta$ is expansive. Let $E \subseteq \Z$ be the observation window confirming expansivity, that is, a finite domain such that for any $c_1,c_2\in \Sigma^\Z$, $c_1 \neq c_2$, there is $n\in \Z$ such that $H_\theta ^n (c_2) \notin \cyl (H_\theta ^n (c_1),E)$. Let $D\subseteq \Z$ be a finite domain and $c_1 \in \Sigma^\Z$ a configuration. It is enough to show that for some $c_2 \in \cyl (c_1,D)$, $m \in \Z_+$ and $x \in E$, we have $H_\theta^m (c_1)(x) \neq H_\theta^m (c_2)(x)$.

By expansivity there is $n\in \Z$ such that $H_\theta^n (c_1)(x) \neq H_\theta^n (c_2)(x)$ for some $x\in E$. If $n\geq 0$ the proof is complete, so assume $n<0$. Consider the space $(\Sigma \times \Sigma)^\Z$ and \textsc{nuca} $H_\phi:(\Sigma \times \Sigma)^\Z \rightarrow (\Sigma \times \Sigma)^\Z$ such that for $e \in (\Sigma \times \Sigma)^\Z$ and $e_1,e_2 \in \Sigma^\Z$ where $e(y) = (e_1(y),e_2(y))$ for all $y\in \Z$, $H_\phi$ maps $e$ such that
\begin{align*}
H_\phi(e)(y) = (H_\theta (e_1)(y),H_\theta (e_2)(y)),
\end{align*}
for all $y\in \Z$. Obviously such $e_1,e_2$ exist for any $e$, and $H_\phi$ is a \textsc{nuca}. 

Because $H_\theta$ is bijective, clearly $H_\phi$ is too. Then by Lemma \ref{poincare}, temporally recurrent configurations of $H_\phi$ are dense. Then there's a temporally recurrent $e \in (\Sigma \times \Sigma)^\Z$ with $e_1,e_2 \in \cyl (c, D)$ such that $e(y) = (e_1(y),e_2(y))$ for all $y \in \Z$ and $H_\theta^n (e_1)(x) = H_\theta^n (c_1)(x)$, $H_\theta^n (e_2)(x) = H_\theta^n (c_2)(x)$ by the continuity of $H_\theta$. Because $e$ is temporally recurrent, there is some $m>0$ such that $H_\theta^m (e_1)(x) = H_\theta^n (e_1)(x)$ and $H_\theta^m (e_2)(x) = H_\theta^n (e_2)(x)$. Then because 
\begin{align*}
H_\theta^m (e_1)(x) = H_\theta^n (c_1)(x) \neq H_\theta^n (c_2)(x) = H_\theta^m (e_2)(x),
\end{align*}
either $H_\theta^m (e_1)(x) \neq H_\theta^m (c_1)(x)$ or $H_\theta^m (e_2)(x) \neq H_\theta^m (c_1)(x)$. Therefore $H_\theta$ is sensitive.
\end{proof}

\section{Conclusions}

In summary, we've shown that a bijective \textsc{nuca} with a uniformly recurrent rule distribution is reversible and that this is not true of bijective \textsc{nuca} in general, even with recurrent rule distributions. We've shown that a bijective \textsc{nuca} is balanced, and that a surjective \textsc{nuca} with a recurrent a rule distribution is balanced, once again not true of all surjective \textsc{nuca}. 

We've also presented counter-examples for the usual dichotomy of sensitivity and equicontinuity in 1D \textsc{ca}, that is, a 1D \textsc{nuca} which is not sensitive and has no equicontinuity points, and one which has a non-residual and non-empty set of equicontinuity points. The question of whether there are such counter-examples with (uniformly) recurrent rule distributions remains open.  Finally, we've shown that (positively) expansive \textsc{nuca} are sensitive.

\subsubsection{Acknowledgements} 

We would like to acknowledge funding from the Magnus Ehrnrooth foundation and the Research Council of Finland, grant 354965.

%
%
%
%

\end{document}